\documentclass [a4paper]{amsart}

\usepackage{amsmath}
\usepackage{amsthm}
\usepackage{amssymb}
\usepackage{verbatim}
\usepackage{ifthen}
\usepackage{amsfonts}
\usepackage[mathscr]{euscript}
\usepackage{xypic}      
\usepackage{graphicx}
\usepackage{textcomp}
\usepackage{yhmath}
\usepackage{bbm}
\usepackage{empheq}
\usepackage[bookmarks=true, linkbordercolor={0.4 0.35 0.8}]{hyperref}

\newtheorem{theorem}{Theorem}[section]
\newtheorem{corollary}[theorem]{Corollary}
\newtheorem{lemma}[theorem]{Lemma}

\theoremstyle{definition}
\newtheorem{definition}[theorem]{Definition}
\newtheorem{example}[theorem]{Example}

\theoremstyle{remark}
\newtheorem{remark}[theorem]{Remark}

\numberwithin{equation}{section}

\newsavebox{\notationlistcontent}

\newcommand{\scr}{\mathcal}

\DeclareMathOperator{\Symb}{Symb}
\DeclareMathOperator{\Vol}{Vol}

\newcommand{\st}{\; | \;}

\newcommand{\defeq}{:=}

\newcommand{\CC}{\mathbb{C}}

\newcommand{\NN}{\mathbb{N}}
\newcommand{\RR}{\mathbb{R}}

\newcommand{\ip}[1]{\langle #1 \rangle}

\newcommand{\into}{\hookrightarrow}

\newcommand{\pddiff}[3][.]{\ifthenelse{\equal{#1}{.}}
        {\frac{\partial^2 #2}{\partial #3^2}}
        {\frac{\partial^2 #1}{\partial #2 \partial #3}}}

\newcommand{\Lie}{\mathsf}
\newcommand{\lie}{\mathfrak}

\newcommand{\Univ}[2][.]{\ifthenelse{\equal{#1}{.}}{\mathcal{U}}{\mathcal{U}^{(#1)}}(\mathfrak{#2})}

\DeclareMathOperator{\SL}{SL}
\DeclareMathOperator{\SU}{SU}
\DeclareMathOperator{\Sp}{Sp}

\newcommand{\supp}[1][.]{%
  \ifthenelse{\equal{#1}{.}}
    {}
    {#1\text{-}}
  \mathrm{supp}}

\newcommand{\PsiDOc}[1]{\Psi^{#1}_c}
\newcommand{\PsiDO}[1]{\Psi^{#1}}
\newcommand{\Cs}{\overline{\Psi^{-\infty}_c}}
\newcommand{\cosphere}[1]{S#1}
\newcommand{\fibration}{\tau}
\newcommand{\Op}[2][]{\mathrm{Op}_{\Lie{#1}}#2}
\newcommand{\dHaar}[1]{d_{\Lie{#1}}}
\newcommand{\Haar}[1]{\mu_{\Lie{#1}}}
\newcommand{\modular}[1]{\Delta_{\Lie{#1}}}
\newcommand{\chart}{\varphi}
\newcommand{\Grpd}{\scr{G}}
\newcommand{\crit}{C}

\newcommand{\Roots}{\Sigma}
\newcommand{\Simpleroots}{\Pi}

\begin{document}

\title{Foliation $C^*$-algebras on multiply fibred manifolds}

\author{Robert Yuncken}
\email{yuncken@math.univ-bpclermont.fr}

\subjclass[2000]{58J40 (primary); 22E30, 43A85 (secondary)}

\begin{abstract}
Motivated by index theory for semisimple groups, we study the relationship between the foliation $C^*$-algebras on manifolds admitting multiple fibrations.  Let $\scr{F}_1,\ldots,\scr{F}_r$ be a collection of smooth foliations of a manifold $\scr{X}$.  We impose a condition of local homegeneity on these foliations which ensures that they generate a foliation $\scr{F}$ under Lie bracket of tangential vector fields.  We then show that the product of longitudinal smoothing operators $\Psi_c^{-\infty}(\scr{F}_1) \cdots \Psi_c^{-\infty}(\scr{F}_r)$ belongs to the $C^*$-closure of $\Psi_c^{-\infty}(\scr{F})$.  An application to noncommutative harmonic analysis on compact Lie groups is presented.
\end{abstract}

\maketitle

% -----------------------------------------------------------------

\small \noindent \textbf{2000 MSC Codes:} 58J40 (primary), 22E30, 43A85 (secondary).\\
\noindent \textbf{Keywords:}  pseudodifferential operators; foliation algebras; Lie groups; flag varieties.

\normalsize

\vspace{1ex}

\section{Introduction}
\label{sec:introduction}

A major thread throughout index theory is the study of longitudinal pseudodifferential operators---operators which (pseudo)differentiate along the leaves of a foliation.  This stream of ideas was already begun in the foundational papers of Atiyah and Singer \cite{Atiyah-Singer4} (for fibred manifolds), and generalized radically by Connes \cite{Connes:survey} (for foliations).  However, to date all of the literature has concentrated on manifolds with a single foliation.  The purpose of this paper is to begin to study foliation algebras on manifolds with multiple foliations.

It is important to understand the motivation here, which comes from very specific examples in representation theory.  Recall that for semisimple Lie groups, much of the representation theory centres on the generalized principal series representations, which act on section spaces of line bundles over the flag manifold $\scr{X}\defeq\Lie{G/B}$ of the group $\Lie{G}$.  Intertwining operators between these representations occur naturally as pseudodifferential operators along the various fibrations of $\scr{X}$.  

It is the analysis of these operators that we are really interested in.  Our decision to use the language of foliations rather than fibrations in this paper was made only to simplify the statements and proofs of the results, not to anticipate generalizations away from these fundamental examples.\footnote{Having said that, one could certainly conceive of applications to parabolic geometries, following \cite{CSS}.}

In \cite{Yuncken:BGG}, we demonstrated how the analysis of these operators can be applied to index theory.  Specifically, we used the BGG-complex of $\SL(3,\CC)$ to construct an explicit model for Kasparov's $\gamma$ element as the image of an element of the equivariant $K$-homology of the flag variety $K_G(\scr{X})$.  This parallels earlier constructions by Kasparov and Julg (\cite{Kasparov:Lorentz, Julg-Kasparov, Julg}) which were used to prove the Baum-Connes conjecture for all discrete subgroups of rank one Lie groups\footnote{The case of $\Sp(n,1)$ was first proven in \cite{Lafforgue:Banach_KK} using slightly different methods.}.  The conjecture remains open in rank greater than one.

Central to the above mentioned construction of $\gamma$ for $\SL(3,\CC)$ was a compactness theorem for products of negative order pseudodifferential operators along the fibrations of $\scr{X}$ (\cite{Yuncken:PsiDOs_on_SLnC}).  The proof used some highly nontrivial computations in noncommutative harmonic analysis for the maximal compact subgroup $\SU(3)$.  While that argument could, in principle, be generalized to all compact semisimple groups, in practice the computations become overwhelming.

In this paper, we change our approach by using noncommutative harmonic analysis in the sense of M.~Taylor \cite{Taylor:microlocal}.  This allows results to be proven in broad generality---in particular for any generalized flag manifold.  Moreover, the proofs become considerably more enlightening.

\bigskip

Let us state the main results.  Let $\scr{X}$ be a smooth manifold equipped with a collection of $r$ smooth foliations $\scr{F}_1,\ldots,\scr{F}_r$.  Let $\PsiDOc{-\infty}(\scr{F}_j)$ denote the set of longitudinally smoothing operators along $\scr{F}_j$ with compact support.  We shall explain exactly what we mean by this in Section \ref{sec:PsiDOs_on_groups}.  With the appropriate definition, these act as bounded operators on $L^2\scr{X}$, and their norm-closure $\Cs(\scr{F}_j)$ is a $C^*$-algebra.  It contains the order $-d$ longitudinal pseudodifferential operators $\PsiDOc{-d}(\scr{F}_j)$ for any $-\infty\leq-d <0$.

We shall control the relative geometry of the foliations by assuming the following local homogeneity condition.  Let $\Lie{G}$ be a connected Lie group, and let $\lie{h}_1,\ldots,\lie{h}_r$ be a collection of Lie subalgebras of its Lie algebra $\lie{g}$.  We let $\scr{H}_j$ denote the foliation of $\Lie{G}$ generated by left translates of $\lie{h}_j$.  

\begin{definition}
\label{def:local_homogeneity}
The family of foliations $\scr{F}_1,\ldots,\scr{F}_r$ will be called {\em locally homogeneous} (with structural data $\lie{h}_1,\ldots,\lie{h}_r \leq \lie{g}$) if there is an atlas of local charts $\phi_\alpha: U_\alpha \to \scr{X}$ with $U_\alpha \subseteq \Lie{G}$ such that $d\phi_\alpha$ maps each $\scr{H}_j$ isomorphically to $\scr{F}_j$ on its domain.
\end{definition}

This condition ensures that the set of vector fields generated by $C^\infty(T\scr{F}_1), \ldots, \linebreak C^\infty(T\scr{F}_r)$ via Lie brackets is itself the space of vector fields tangent to a foliation, which we denote by $\scr{F}$.   In each chart, $T\scr{F}$ is the bundle of left translates of the Lie algebra generated by $\lie{h}_1,\ldots,\lie{h}_r$.

\begin{theorem}
\label{thm:main_theorem}
Let $\scr{F}_1,\ldots,\scr{F}_r$ be a locally homogeneous family of foliations.  With $\scr{F}$ as above,
$$
  \Cs(\scr{F}_1) \cdots \Cs(\scr{F}_r) \subseteq \Cs(\scr{F}).
$$
\end{theorem}

Thus, successive smoothing along the directions of $\scr{F}_1,\ldots,\scr{F}_r$ yields an operator which is almost smoothing, not just in the directions spanned by the $T\scr{F}_j$, but in all directions generated from them via Lie brackets.

If $\scr{F}$ is the foliation of $\scr{X}$ by a single leaf then $\Cs(\scr{F}) = \scr{K}(L^2\scr{X})$.  Thus, we have the following important corollary.

\begin{definition}
\label{def:Hormander_condition}
We shall say the foliations $\scr{F}_1,\ldots,\scr{F}_r$ satisfy {\em H\"ormander's condition} if the Lie algebra of all smooth vector fields on $\scr{X}$ is generated by $C^\infty(T\scr{F}_1),\ldots, \linebreak C^\infty(T\scr{F}_r)$. 
\end{definition}

For a locally homogeneous family of foliations, this is equivalent to requiring that $\lie{h}_1,\ldots,\lie{h}_r$ generate $\lie{g}$ as a Lie algbera.

\begin{corollary}
\label{cor:main_corollary}
Let $\scr{F}_1,\ldots, \scr{F}_r$ be a locally homogeneous family of foliations which satisfy H\"ormander's condition.  If $A_j\in\PsiDOc{-1}(\scr{F}_j)$ for each $j$, then the product $A_1\cdots A_r$ is a compact operator.
\end{corollary}

\bigskip

Applications to equivariant index theory require further results, which we shall defer to a subsequent paper.  For now, we will provide a brief application to noncommutative harmonic analysis on compact Lie groups.  

Suppose that $\Lie{K}$ is a compact Lie group, and $\Lie{K}_1$, $\Lie{K}_2$ are closed subgroups which generate $\Lie{K}$.  Let $U$ be a unitary representation of $\Lie{K}$ on a Hilbert space $H$ for which all irreducible $\Lie{K}$-types have finite multiplicity.\  If $\pi_1$, $\pi_2$ are irreducible representations for $\Lie{K}_1$ and $\Lie{K}_2$, respectively, then the $\pi_1$- and $\pi_2$-isotypical subspaces of $H$ are essentially orthogonal, in the sense that they have arbitrarily small inner products on the complement of some finite-dimensional subspace. (See Theorem \ref{thm:essential_orthotypicality} for a precise statement.)

\bigskip

The paper is organized as follows.  In Section \ref{sec:PsiDOs_on_groups} we provide some background on noncommutative microlocal analysis as it pertains to longitudinal pseudodifferential operators on Lie groups.  The technical heart of the paper is Section \ref{sec:singular_coordinates}, in which we prove results about integral operators in nonsingular coordinate systems.  In Sections \ref{sec:groups} and \ref{sec:local} we apply this analysis to longitudinal pseudodifferential operators.  The final section describes the above-mentioned application to noncommutative harmonic analysis.

We would like to thank N.~Higson, C.~Debord and J.-M.~Lescure for helpful conversations.

\subsection{Notation}

Throughout, $\Lie{G}$ will denote a connected Lie group, and $\Lie{H}_1,\ldots,\Lie{H}_r$ closed subgroups.  For any Lie group, we shall use the corresponding Fraktur letter to denote its Lie algebra, often without mention.  We will denote left Haar measure on $\Lie{G}$ by $\Haar{G}$ or $\dHaar{G}x$, and the modular function by $\modular{G}$.    

%For a Lie subalgebra $\lie{h}\leq \lie{g}$, we shall use $\ldist{\lie{h}}$ to denote the subbundle of the tangent bundle $T\Lie{G}$ obtained by left-translation of $\lie{h}$.  Thus, if $\Lie{H}=\mathrm{Exp}(\lie{h})$ is a closed subgroup, $\ldist{\lie{h}}$ is the bundle of vectors tangent to the coset fibration $\Lie{G} \to \Lie{G/H}$.

%For a function on a product space $f:M \times N \to \CC$, we shall use $\supp[M](f)$ and $\supp[N](f)$ to denote the projection of the support of $f$ onto $M$ and $N$, respectively.

% ----------------------------------------------------------------

\section{Longitudinal pseudodifferential operators}
\label{sec:PsiDOs_on_groups}

Let $\scr{F}$ be a smooth foliation of a manifold $\scr{X}$.  If $\scr{F}$ is the tangent bundle to a smooth fibration $p:\scr{X}\to\scr{Y}$ then elements of $\PsiDO{d}(\scr{F})$ can be defined as smooth families of pseudodifferential operators of order $d$ on the fibres, as in \cite{Atiyah-Singer4}.  Then $\PsiDOc{d}(\scr{F})$ will be the subset of those whose distributional kernel has compact support in $\scr{X}\times\scr{X}$.  However, in what follows it will be convenient if we don't have to continually assume that our foliations are fibrations.  For this, some brief technical remarks are in order to clarify our definitions.

As Connes observed, the correct home for pseudodifferential operators on foliations is the holonomy groupoid $\Grpd \defeq \Grpd(\scr{X},\scr{F})$. (See \cite{Connes:NCG, Connes:survey} for the definitions.)  The smooth convolution algebra $C_c^\infty(\Grpd)$ is naturally represented on the $L^2$-spaces of the leaves, and the resulting $C^*$-algebra is Connes' foliation algebra $C^*(\scr{X},\scr{F})$.  

However, in order to make sense of statements such as Theorem \ref{thm:main_theorem} above, we need the various convolution algebras to be all represented on a common Hilbert space.  Thus, in this paper, we will be working with what J.~Roe \cite{Roe:foliations} refers to as the ``global representation'' of $C_c^\infty(\Grpd)$ on $L^2(\scr{X})$.  Specifically, for $k\in C_c^\infty(\Grpd)$, $f\in L^2(\scr{X})$, set
$$
  k\cdot f(x) \defeq \int_{\Grpd_x} k(\gamma) f( r(\gamma)) \,d\gamma.
$$
where, $r,s$ are the range and source maps of $\Grpd$ and $\Grpd_x \defeq s^{-1}(x)$.  It is the image of this representation in $\scr{B}(L^2\scr{X})$ which we refer to as $\PsiDOc{-\infty}(\scr{F})$.  Note that this will not generally extend to a representation of $C^*(\Grpd)$, unless the transverse component of Riemannian measure is holonomy invariant (as is the case for fibrations).

\medskip

In any case, we denote by $\Cs(\scr{F})$ the norm-closure of $\PsiDOc{-\infty}(\scr{F})$ in $\scr{B}(L^2(\scr{X}))$. If $\scr{F}$ comes from a fibration, $\Cs(\scr{F}) \cong C^*_r(\scr{X},\scr{F})$.

We also remark that $\PsiDOc{-d}(\scr{F})$ is dense in $\Cs(\scr{F})$ for any $-\infty\leq d <0$.  Thus, our $C^*$-algebraic approach destroys any notion of order of (pseudo-)\linebreak differentiation and hence precludes all of the subtle analytic estimates that a full pseudodifferential calculus affords.  On the other hand, this norm-density allows us to work with the relatively simple class of longitudinally smoothing operators, and still our results remain strong enough for applications to index theory, which is concerned with much coarser analytic properties (Fredholmness, compactness, {\em etc}).

\bigskip

We now specialize to Lie groups, where we will reformulate longitudinal smoothing operators following the pattern of \cite{Taylor:microlocal}.

Let $\Lie{G}$ be a connected Lie group, and let $\Lie{H}$ be a connected Lie subgroup, not necessarily closed.  We endow $\Lie{H}$ with the topology associated to its intrinsic smooth structure, not the subspace topology.  Let $\scr{H}$ denote the foliation of $\Lie{G}$ by left-cosets of $\Lie{H}$. 

A longitudinally smoothing operator $K\in\PsiDOc{-\infty}(\scr{H})$ is given by an integral formula
\begin{equation}
\label{eq:Op_k}
  Ku(x) \defeq \int_{\Lie{H}} k(x,h) u(h^{-1}x) \, dh,
\end{equation}
where $k \in C_c^\infty(\Lie{G}\times\Lie{H})$.  We will sometimes denote such an operator by $\Op[H]{k}$.

Define $L^2(\Lie{G})$ with respect to left-invariant Haar measure.  The adjoint of $\Op[H]{k}$ is $\Op[H](k^*)$ where
\begin{equation}
\label{eq:adjoint}
  k^*(x,h) \defeq \overline{k(h^{-1}x,h^{-1})}\, \modular{H}(h).
\end{equation}

% ----------------------------------------------------------------

\section{Integral operators in singular coordinate systems}
\label{sec:singular_coordinates}

Let $\Lie{G}$, $\Lie{H}$ be as in the previous section.  We need to generalize the operators of \eqref{eq:Op_k} by reparameterizing the variable $h\in \Lie{H}$ with a singular change of coordinates, as follows.

Let $M$ be a smooth manifold with smooth measure $dm$ and let $\phi: M \to \Lie{H}$ be a smooth function.  We consider operators $A$ of the form
\begin{equation}
\label{eq:Op_a}
  Au(x) \defeq \int_M a(x,m) u(\phi(m)^{-1}x) \,dm
\end{equation}
for $a\in C_c^\infty(\Lie{G}\times M)$.

\begin{lemma}
\label{lem:norm_estimate}
The formula \eqref{eq:Op_a} defines a bounded operator $A$ on $L^2(\Lie{G})$ with norm $\|A\| \leq \|a\|_\infty \Vol(\supp[M](a))$.

\end{lemma}

\begin{proof}
Let $L_g$ denote the left regular representation of $g\in\Lie{G}$ on $L^2\Lie{G}$.  We can write $A = \int_M A_m \,dm$ where $A_m$ is the operator $u \mapsto a(\cdot,m) L_{\phi(m)} u$.  Note that $\|A_m\| \leq \|a\|_\infty$, and that $A_m=0$ for $m\notin\supp[M](a)$.  The result follows.
\end{proof}

Define the {\em critical set} of $\phi: M \to \Lie{H}$ as $\crit(\phi) \defeq \{m\in M \st D\phi(m)\text{ is not onto}\}$.  The $M$-support of $a\in C_c^\infty(\Lie{G}\times M)$, denoted $\supp[M](a)$, is the projection of the support of $a$ onto $M$. 

\begin{lemma}
\label{lem:supported_on_regular_set}
Suppose that $\supp[M](a) \cap \crit(\phi) = \emptyset$.  Then the operator $A$ of \eqref{eq:Op_a} is in $\Psi_c^{-\infty}(\scr{H})$.
\end{lemma}

\begin{proof}
Let $d\defeq \dim(\Lie{H})$ and $n\defeq \dim (M)$.  Note that if $n<d$ then $\crit(\phi)=M$, so necessarily $n\geq d$.
By using a partition of unity subordinate to local charts on $M$, we may reduce to the case where $M$ is a bounded open subset of $\RR^n$. 

Let $N\defeq {n \choose d}$.  Let $E_1,\ldots, E_N$ denote the coordinate $d$-planes of $\RR^n$ (in any order).  For each of these, we will define a Jacobian of $\phi$ in the spirit of the Implicit Function Theorem.  Thus, let $p_i$ denote the orthogonal projection of $\RR^n$ onto $E_i^{\perp}$, and augment $\phi$ to the map
$$
  \Phi_i : M \to \Lie{H}\times E_i^\perp; \quad m \mapsto (\phi(m), p_i(m)).
$$
This is a local diffeomorphism at $m$ if and only if $D\phi(m)|_{E_i}$ is onto.  Define $J_i$ as the Radon-Nikodym derivative of $\Phi_i$:
\begin{equation}
\label{eq:Jacobian}
  J_i(m) \defeq \frac{{\Phi_i}_*(dm)}{\dHaar{H} x \, de'},
\end{equation}
where $de'$ is Lebesgue measure on $E_i^\perp$.

Let 
\begin{equation}
\label{eq:Jacobians}
\mathbf{J} = (J_1,\ldots,J_N):M\to\RR^N.
\end{equation}  
By the hypothesis of the Lemma, there is some $\delta>0$ such that the open sets $U_i \defeq J_i^{-1}(\delta,\infty)$ cover $\supp[M](a)$.  Choose $\psi_i \in C_c^\infty(U_i)$ such that $\sum_i \psi_i \equiv 1$ on $\supp[M](a)$.   Put $a_i(x,m)\defeq \psi_i(m)a(x,m)$, so that $A=\sum_iA_i$ where
\begin{equation}
\label{eq:Op_ai}
  A_iu(x) \defeq \int_{U_i} a_i(x,m) u(\phi(m)^{-1}x) \,dm.
\end{equation}
%We will now show that $A_i \in \PsiDOc{-\infty}(\scr{H})$ for each $i$. 

Fix $i\in\{1,\ldots,N\}$.   By design, $\Phi_i$ is a local diffeomorphism on $\supp[M](a_i)$, so we can find  a finite cover $\{V_j\}$ of $\supp[M](a_i)$ by relatively compact open sets on which $\Phi_i$ is a diffeomorphism to its range.  We write $\Phi_{ij}\defeq\Phi_i|_{V_j}$ for these diffeomorphisms.
Now let $\chi_j\in C^\infty(V_j)$ be a partition of unity subordinate to $\{V_j\}$ and put $a_{ij}(x,m) \defeq \chi_j(m) a_i(x,m)$.  Then
\begin{eqnarray}  
  A_i u(x) &=& \sum_j \int_{V_j}  a_{ij}(x,m)\, u(\phi(m)^{-1} x) \,dm . \nonumber \\
    &=& \sum_j \int_{(h,e') \in \Phi_i(V_j)} a_{ij}(x,\Phi_{ij}^{-1}(h,e')) \, u(h^{-1}x)\,
      J_i(\Phi_{ij}^{-1}(h,e'))^{-1} \,dh \,de' \nonumber \\
    &=& \sum_j \int_{(h,e') \in \Phi_i(V_j)}\, k_{ij}(x,h,e') \,u(h^{-1}x) \,dh \,de' , 
      \label{eq:local_Op_a}
\end{eqnarray}
where
$$
  k_{ij}(x,h,e') \defeq a_{ij}(x,\Phi_{ij}^{-1}(h,e'))  \,J_i(\Phi_{ij}^{-1}(h,e'))^{-1} .
$$
Since $J_i(m) \geq \delta$ on $\supp[M](a_{ij})$,  $k_{ij}(x,h,e')$ extends to a smooth compactly supported function on $\Lie{G}\times\Lie{H}\times E_i^\perp$.  We obtain
$$
  A_i u(x) = \int_{\Lie{G}}\left(\sum_j \int_{E_i^\perp} k_{ij}(x,h,e') \,de' \right) u(h^{-1}x) \,dh.
$$
The quantity in parentheses is a smooth compactly supported function of $(x,h)\in\Lie{G\times H}$, so $A_i\in \PsiDOc{-\infty}(\scr{H})$.  This completes the proof.

\end{proof}

\begin{corollary}
\label{cor:multiplier}
Any $A$ of the form \eqref{eq:Op_a} is a multiplier of the $C^*$-algebra $\Cs(\scr{H})$.
\end{corollary}

\begin{proof}
Let $K \in \Psi_c^{-\infty}(\scr{H})$ be given in the form \eqref{eq:Op_k}, for some $k\in C_c^\infty(\Lie{G}\times\Lie{H})$.  Then
$$
  AKu(x) = \int_{M\times\Lie{H}} a(x,m)\, k( \phi(m)^{-1}x, h) \, u((\phi(m)h)^{-1}x)\, dm\,dh.
$$
The map
$$
  M\times\Lie{H} \to \Lie{H}; \quad (m,h) \mapsto \phi(m)h
$$
is a submersion, so by the above lemma $AK\in\Psi^{-\infty}(\scr{H})$.  Similarly, $KA\in\Psi^{-\infty}(\scr{H})$.  A density argument completes the proof.
\end{proof}

\begin{theorem}
\label{thm:Op_a_is_compact}
Let $a\in C_c^\infty(\Lie{G}\times M)$ and $\phi:M\to\Lie{G}$ be smooth.  If $\crit(\phi)$ has measure zero, then the operator $A$ of Equation \eqref{eq:Op_a}
%\begin{equation*}
%  Au(x) \defeq \int_M a(x,m) u(\phi(m)^{-1}x) \,dm
%\end{equation*}
is in $\Cs(\scr{H})$.
\end{theorem}

\begin{proof}
Fix $\epsilon>0$.  Choose an open neighbourhood $U$ of $\crit(\phi)$ with measure less than $\epsilon$.   Let $\chi_1, \chi_2$ be a smooth partition of unity on $M$ with $\supp(\chi_1) \subset U$ and $\supp(\chi_2) \subset M\setminus\crit(\phi)$.  Then $A=A_1+A_2$ with
$$
  A_i u(x) \defeq \int_M \chi_i(m) \,a(x,m) \,u(\phi(m)^{-1}x) \,dm .
$$
Now, $A_2 \in \Psi_c^{-\infty}(\scr{H})$ by Lemma \ref{lem:supported_on_regular_set}, and $\|A_1\| < \epsilon$ by Lemma \ref{lem:norm_estimate}.

\end{proof}

Recall that every Lie group admits a canonical real-analytic structure.  In practice, it will be real-analyticity that ensures the measure-zero critical set required for Theorem \ref{thm:Op_a_is_compact}.

\begin{corollary}
\label{cor:analytic_case}
Let $M$ be a connected real-analytic manifold, and $\phi:M\to\Lie{H}$ be a real-analytic map with image of nonzero measure, then the operator $A$ of \eqref{eq:Op_a} is in $\Cs(\scr{H})$.
\end{corollary}

\begin{proof}
As above, let $d\defeq\dim(\Lie{H})$, $n\defeq \dim(M)$ and $N \defeq {n \choose d}$.  
In any analytic chart $U$ of $M$, the critical set of $\phi$ is the zero set of the real-analytic function $\mathbf{J}:U\to\RR^N$ of Equation \eqref{eq:Jacobians}.  This function is not everywhere zero, by Sard's Theorem.  Real-analyticity implies $\mathbf{J}^{-1}(\mathbf{0})$ has measure zero.

\end{proof}

% --------------------------------------------------------------------------

\section{Products of longitudinal pseudodifferential operators on Lie groups}
\label{sec:groups}

Let $\Lie{H}_1,\ldots,\Lie{H}_r$ be connected Lie subgroups (not necessarily closed) of the connected Lie group $\Lie{G}$.  Let $\Lie{H}$ denote the subgroup they generate:
$$
  \Lie{H} \defeq \{ x_1 x_2 \cdots x_k \st \text{Each $x_i$ is in some $\Lie{H}_j$}\}.
$$
This is a connected Lie subgroup whose Lie algebra $\lie{h}$ is the Lie algebra generated by $\lie{h}_1,\ldots,{h}_r$.  We use $\scr{H}$ (resp.~$\scr{H}_j$) to denote the foliation of $\Lie{G}$ by left-cosets of $\Lie{H}$ (resp.~$\Lie{H}_j$).

\begin{theorem}
\label{thm:compact_product}
With the above notation, 
$$
  \Cs(\scr{H}_1) \cdots \Cs(\scr{H}_r) \subseteq \Cs(\scr{H})
$$
\end{theorem}

\begin{proof}
Let $\Lie{H}_1 \Lie{H}_2 \cdots \Lie{H}_r$ denote the set of products $\{h_1h_2\cdots h_r \st h_j\in\Lie{H}_j \text{ for all } j \}$. We first prove the theorem under the assumption that $\Lie{H}_1 \Lie{H}_2 \cdots \Lie{H}_r$ has nonzero measure in $\Lie{H}$.

Let $K_j = \Op[H](k_j)\in\Psi^{-\infty}_c(\scr{H}_j)$, with $k_j\in C_c^\infty(\Lie{G}\times\Lie{H})$.   By iterating equation \eqref{eq:Op_k}, we see that
$$
  K_1 K_2\cdots K_r u(x) = \int_{\Lie{H}_1\times\cdots\times\Lie{H}_r} 
    a(x,h_1,\ldots,h_r) \,u((h_1\cdots h_r)^{-1} x) 
    \,dh_1\cdots \,dh_r,
$$
where
$$
  a(x,h_1,\ldots,h_r) \defeq \prod_{j=1}^r k_j((h_1\cdots h_{j-1})^{-1}x,h_j) .
$$
Here, $a\in C_c^\infty(\Lie{G}\times(\Lie{H}_1\times\cdots\times\Lie{H}_r))$ and the map $(h_1,\ldots,h_r)\mapsto h_1\cdots h_r$ is real-analytic, so Corollary \ref{cor:analytic_case} implies that $K_1\cdots K_r\in\Cs(\scr{H})$.  
%Since $\Psi^{-\infty}_c(\scr{H}_j)$ is dense in $\Cs(\scr{H}_j)$, the result follows.

Now we drop the assumption on the product $\Lie{H}_1\cdots \Lie{H}_r$.  Note that $K_1\cdots K_r$ is in the multiplier algebra of $\Cs(\scr{H})$, by Corollary \ref{cor:multiplier}.
Since $\Lie{H}_1,\ldots,\Lie{H}_r$ generate $\Lie{G}$, there is some $n\in\NN$ for which the set 
$$
  \underbrace{
    (\Lie{H}_1 \Lie{H}_2 \cdots \Lie{H}_r \Lie{H}_r \cdots \Lie{H}_2  \Lie{H}_1) 
    (\Lie{H}_1 \Lie{H}_2 \cdots \Lie{H}_r \Lie{H}_r \cdots \Lie{H}_2  \Lie{H}_1) 
    \cdots
    (\Lie{H}_1 \Lie{H}_2 \cdots \Lie{H}_r \Lie{H}_r \cdots \Lie{H}_2  \Lie{H}_1) 
  }_\text{$n$ copies}
$$
has positive measure in $\Lie{G}$.  By employing the formula \eqref{eq:adjoint} for the adjoint of $K_j$, the previous argument shows that $(K_1\cdots K_r K_r^* \cdots K_1^*)^n \in \Cs(\scr{H})$.  But now standard $C^*$-algebra theory implies that $K_1\cdots K_r$ is in $\Cs(\scr{H})$ (see \cite[I.5.3]{Davidson}).

\end{proof}

%\begin{remark}
%Theorem \ref{thm:compact_product} remains true for non-connected $\Lie{G}$, since each operator in $\PsiDOc{-1}(\scr{H}_j)$ will be supported on only finitely many connected components.  
%\end{remark}

% ------------------------------------------------------------------

\section{Locally homogeneous structures}
\label{sec:local}

We now pass to manifolds which are locally modelled on Lie groups.  We continue with the above notation: $\lie{g}$ is the Lie algebra of a connected Lie group $\Lie{G}$; $\lie{h}_1,\ldots,\lie{h}_r$ is a family of Lie subalgebras and $\lie{h}$ is the Lie algebra they generate;  $\scr{H}_1,\ldots,\scr{H}_r$ and $\scr{H}$ are the corresponding left-coset foliations.

Recall (Definition \ref{def:local_homogeneity}) that a family of foliations $\scr{F}_1, \ldots \scr{F}_r$ of a manifold $\scr{X}$ is called {\em locally homogeneous} if $\scr{X}$ admits an atlas of local diffeomorphisms from $\Lie{G}$, under which the fibrations map to $\scr{H}_1,\ldots,\scr{H}_r$.  As described in the introduction, such a family $\scr{F}_1, \ldots, \scr{F}_r$ generates a foliation $\scr{F}\subseteq T\scr{X}$, which in each chart is $\scr{H}$.

Theorem \ref{thm:main_theorem} is then an immediate consequence of Theorem \ref{thm:compact_product} via a partition of unity argument.

\bigskip

We now apply this to the key example of generalized flag varieties.

\begin{example}
\label{ex:flag_variety}

Let $\Lie{G}$ be a complex semisimple Lie group of rank $r$.  Fix a Cartan subalgebra $\lie{h}$.  Fix a system of positive roots $\Roots^+$, with simple roots $\Simpleroots$.   Let $\lie{n} \defeq \bigoplus_{\alpha\in\Roots^+}\lie{g}_{\alpha}$, $\overline{\lie{n}} \defeq \bigoplus_{\alpha\in\Roots^+}\lie{g}_{-\alpha}$, and let $\overline{\lie{b}}\defeq\lie{h}\oplus\overline{\lie{n}}$ be the `lower' Borel subalgebra.  The corresponding Lie groups will be given by upper case letters.

Let $\scr{X}\defeq\Lie{G}/\overline{\Lie{B}}$ be the flag variety of $\Lie{G}$.  For any $S \subseteq \Simpleroots$, let $\ip{S}$ denote the set of positive roots spanned by $S$, {\em ie},
$$
  \ip{S} \defeq \{ \alpha\in\Roots^+ \st \alpha = \sum_{\beta\in S} n_\beta \beta \text{ for some } n_\beta \in \NN \}. 
$$
Let $\lie{n}_S \defeq \bigoplus_{\alpha\in\ip{S}} \lie{g}_{\alpha}$.  Let $\overline{\lie{p}}_S$ denote the parabolic subalgebra $\overline{\lie{b}}\oplus \lie{n}_S$. 
Let $\scr{Y}_S \defeq \Lie{G}  / \overline{\Lie{P}}_S$ be the corresponding partial flag variety.  Let $\scr{F}_S$ denote the fibration of $\scr{X}$ by fibres of the quotient map $\fibration_S:\scr{X} \to \scr{Y}_S$.

The map $\chart:\Lie{N} \into \Lie{G} \twoheadrightarrow \Lie{G}/\overline{\Lie{B}} = \scr{X}$ is a diffeomorphism onto its range.  It is clearly $\Lie{N}$-equivariant.  With this as a chart, the fibres of $\fibration_S$ pull back to the left cosets of $\Lie{N}\cap \overline{\Lie{P}}_S = \Lie{N}_S$. 
By taking $\Lie{G}$-translates, we can cover $\scr{X}$ with such charts.  Thus we see that the family of fibrations $(\scr{F}_S)_{S\subseteq\Simpleroots}$ is locally homogeneous with structural data $(\lie{n}_S)_ {S\subseteq\Simpleroots} \leq \lie{n}$.

Theorem \ref{thm:main_theorem} yields the following statement.  Let $S_1,\ldots,S_r \subseteq \Simpleroots$ and put $S = \bigcup_{i=1}^r S_i$.  Then
$$
  \Cs(\scr{F}_{S_1}) \cdots  \Cs(\scr{F}_{S_r}) \subseteq \Cs(\scr{F}_S).
$$
In particular, if $A_\alpha$ is a longitudinal pseudodifferential operator along $\scr{F}_{\{\alpha\}}$, for each $\alpha\in\Simpleroots$, then their product is a compact operator.  This is the generalization of \cite[Theorem 1.3]{Yuncken:PsiDOs_on_SLnC} for arbitrary flag manifolds.

\end{example}

% -----------------------------------------------------------------------------------

\section{Essential orthogonality of sub-types}
\label{sec:essential_orthotypicality}

We conclude with a brief application of the above results in noncommutative harmonic analysis for compact groups.

Let $\Lie{K}$ be a compact Lie group, and let $\Lie{K}_1$ and $\Lie{K}_2$ be closed subgroups.  Let $U:\Lie{K} \to \scr{B}(H)$ be a unitary representation of $\Lie{K}$ with finite multiplicities.  What can be said of the relative position of the $\Lie{K}_1$- and $\Lie{K}_2$-invariant subspaces of $H$, or more generally of the isotypical subspaces?  This question arises naturally in harmonic analysis on flag varieties (see \cite{Yuncken:PsiDOs_on_SLnC}).

Let us introduce some notation.  If $\pi$ is an irreducible representation of a subgroup $\Lie{K}'$ of $\Lie{K}$, $p_\pi$ will denote the orthogonal projection onto $H_\pi$, the $\pi$-isotypical subspace of $U|_{\Lie{K}'}$.   If $S\subseteq \hat{\Lie{K}}'$ is a set of $\Lie{K}'$-types, we put $p_S \defeq \sum_{\pi\in S} p_\pi$.  Note that if $\sigma\in\hat{\Lie{K}}$ and $\pi\in\hat{\Lie{K}}'$, then $p_\sigma$ and $p_\pi$ commute.

We define the {\em inner product} of subspaces $H_1,H_2\leq H$ by
$$
  \ip{H_1,H_2} \defeq \sup \{ \ip{\xi_1,\xi_2} \st \xi_j\in H_j, \|\xi_j\| \leq 1 \}.
$$
%and their {\em angle} as $\angle(H_1,H_2) \defeq \cos^{-1}\ip{H_1,H_2}$.

Suppose to begin with that $\Lie{K} = \Lie{K}_1\times\Lie{K}_2$.  If the representation $H$ has finite $\Lie{K}$-multiplicities, the isotypical subspaces $H_{\pi_1}$ and $H_{\pi_2}$ will have finite dimensional intersection---namely, $H_{\pi_1\otimes\pi_2}$---and moreover will be `perpendicular' in the sense that their respective orthocomplements $H_{\pi_1}\cap (H_{\pi_2})^\perp$ and $H_{\pi_2}\cap (H_{\pi_1})^\perp$ are orthogonal.

Clearly this is not true in generality.  For instance, let $\Lie{K}=\SU(3)$ and $\Lie{K}_1$ and $\Lie{K}_2$ be the subgroups obtained by embedding $\SU(2)$ in the upper-left and lower-right corners of $\Lie{K}$, respectively.  There are infinitely many irreducible $\SU(3)$-representations which contain both a nonzero $\Lie{K}_1$-fixed vector and a nonzero $\Lie{K}_2$-fixed vector, and these fixed vectors are not in general orthogonal.  However, they are asymptotically orthogonal, in the sense that for any $\epsilon>0$, there are only finitely many $\Lie{K}$-types in which the $\Lie{K}_1$- and $\Lie{K}_2$-fixed subspaces have inner product greater than $\epsilon$.

With this example in mind we make the following definition.  We use $V^\sigma$ to denote the vector space underlying an irreducible representation $\sigma\in\hat{\Lie{K}}$.

\begin{definition}
\label{def:essentially_orthotypical}
We say the subgroups $\Lie{K}_1$ and $\Lie{K}_2$ of $\Lie{K}$ are {\em essentially orthotypical} if, for any $\pi_1\in\hat{\Lie{K}}_1$, $\pi_2\in\hat{\Lie{K}}_2$ and $\epsilon>0$, there are only finitely many $\Lie{K}$-types $\sigma\in\hat{\Lie{K}}$ for which $\ip{(V^\sigma)_{\pi_1}, (V^\sigma)_{\pi_2} } \geq \epsilon$.
\end{definition}

An equivalent definition is given by the following lemma.

\begin{lemma}
\label{lem:equiv_defns}
Let $\Lie{K}_1$ and $\Lie{K}_2$ be closed subgroups of a compact Lie group $\Lie{K}$.  The following are equivalent:
\begin{enumerate}
\item $\Lie{K}_1$ and $\Lie{K}_2$ are essentially orthotypical.
\item For any $\pi_1\in\hat{\Lie{K}}_1$ and $\pi_2\in\hat{\Lie{K}}_2$, $p_{\pi_1} p_{\pi_2}$ is a compact operator on every unitary $\Lie{K}$-representation with finite multiplicities.
\end{enumerate}
\end{lemma}

\begin{proof}
\textbf{(i) $\Rightarrow$ (ii):}
Let $S$ be the set of $\sigma\in\hat{\Lie{K}}$ for which $\ip{(V^\sigma)_{\pi_1},(V^\sigma)_{\pi_2}} \geq \epsilon$.  Then on $V^\sigma$ for any $\sigma\notin S$, $\|p_{\pi_1}p_{\pi_2}\|<\epsilon$.  Therefore
$$
  p_{\pi_1}p_{\pi_2} = p_S p_{\pi_1}p_{\pi_2} + p_S^\perp p_{\pi_1}p_{\pi_2}
$$
with the right-hand side being the sum of a finite rank operator and an operator of norm at most $\epsilon$.

\textbf{(ii) $\Rightarrow$ (i):}  Let $\epsilon>0$.  Fix any enumeration $\{\sigma_1, \sigma_2, \ldots\}$ of $\hat{\Lie{K}}$, and let $S_n \defeq \{\sigma_1,\ldots,\sigma_n\}$.  Put $H \defeq \bigoplus_{n=1}^\infty V^{\sigma_n}$.  The projections $p_{S_n}$ on $H$ converge strongly to $1$ as $n\to\infty$, so by the compactness of $p_{\pi_1} p_{\pi_2}$ we have $\|(p_{\pi_1} p_{\pi_2}) p_{S_n}^\perp \| < \epsilon$ for sufficiently large $n$.  

Now let $\sigma$ be any $\Lie{K}$-type not in $S_n$ and suppose $\xi_j \in (V^\sigma)_{\pi_j}$ for $j=1,2$, with $\|\xi_j\|\leq 1$.  After including $V^\sigma$ into $H$, we have
$$ 
 |\ip{\xi_1,\xi_2}| 
    = |\ip{p_{S_n}^\perp p_{\pi_1} \xi_1, p_{S_n}^\perp p_{\pi_2} \xi_2}| 
    = |\ip{ \xi_1, p_{S_n}^\perp p_{\pi_1} p_{\pi_2} \xi_2}| 
    < \epsilon.
$$
Thus $\ip{(V^\sigma)_{\pi_1}, (V^\sigma)_{\pi_2}} <\epsilon$ for all $\sigma\notin S_n$.

\end{proof}

\begin{remark}
\label{rmk:regular representation}
The representation $H=\bigoplus_{n=1}^\infty V^{\sigma_n}$ used in the proof of \textbf{(ii)} $\Rightarrow$ \textbf{(i)} could be replaced by any representation which contains every $\Lie{K}$-type---for instance, the regular representation.
\end{remark}

\begin{theorem}
\label{thm:essential_orthotypicality}

Let $\Lie{K}_1$, $\Lie{K}_2$ be closed subgroups of a compact Lie group $\Lie{K}$.  If $\Lie{K}_1$, $\Lie{K}_2$ generate $\Lie{K}$ then they are essentially orthotypical.  
\end{theorem}

\begin{proof}
Fix $\pi_1\in\hat{\Lie{K}}_1$, $\pi_2\in\hat{\Lie{K}}_2$ and $\epsilon>0$.  Let $\chi_{\pi_j} \in C^\infty(\Lie{K}_j)$ denote the character of $\pi_j$.
Let $U$ be the left-regular representation of $\Lie{K}$ on $L^2(\Lie{K})$.  By the orthogonality of characters,
$$
  p_{\pi_j} f(x) = \int_{\Lie{K}_j} \overline{\chi_{\pi_j}(x)} f(k^{-1}x) \, dx
$$
for any $f\in L^2(\Lie{K})$.  That is, $p_{\pi_j}$ is a longitudinally smoothing operator for the coset fibration of $\Lie{K}_j\leq \Lie{K}$.  Thus, using Remark \ref{rmk:regular_representation}, Lemma \ref {lem:equiv_defns} gives the result.

\end{proof}

It is natural to ask about the converse of Theorem \ref{thm:essential_orthotypicality}.  We will see that the converse holds at least if $\Lie{K}$ is compact semisimple.  

In fact, if the subgroup generated by $\Lie{K}_1$ and $\Lie{K}_2$ is not dense in $\Lie{K}$ then $\Lie{K}_1$ and $\Lie{K}_2$ are not essentially orthotypical.  For consider the the representation of $\Lie{K}$ on $L^2(\Lie{K}/\Lie{K'})$, where $\Lie{K}'$ is the closed subgroup generated by $\Lie{K}_1$ and $\Lie{K}_2$.  The Peter-Weyl Theorem gives the decomposition
$$
  L^2(\Lie{K}/{\Lie{K'}}) \cong \bigoplus_{\sigma\in \hat{\Lie{K}}} V^\sigma \otimes ({V^\sigma}^*)_{\pi_0},
$$
where $\pi_0$ is the trivial representation of ${\Lie{K'}}$.  (On the summands of the right hand side, the $\Lie{K}$-representation is by $\sigma\otimes I$.)
Thus $L^2 (\Lie{K}/{\Lie{K'}})$ has finite $\Lie{K}$-multiplicities.  It also has an infinite dimensional $\Lie{K}'$-fixed subspace, since $V^\sigma$ contains a ${\Lie{K'}}$-fixed vector if and only if ${V^\sigma}^*$ does.  The projections onto the trivial $\Lie{K}_j$-types (for $j=1,2$) both act as the identity on this subspace, so their product cannot be compact.

Compact semisimple groups have no connected dense subgroups \linebreak(\cite{Macias-Virgos}), which gives the converse to Theorem \ref{thm:essential_orthotypicality} for $\Lie{K}$ semisimple.

% ----------------------------------------------------------------

\bibliographystyle{alpha}
\bibliography{microlocal}

\end{document}